\documentclass[12pt]{article}
\usepackage{amsmath,amssymb,amsfonts,amsthm,graphicx,subcaption}
\usepackage[margin=1in]{geometry}

\theoremstyle{plain}

\newtheorem{theorem}{Theorem}[section]
\newtheorem{proposition}[theorem]{Proposition}
\newtheorem{problem}[theorem]{Problem}
\newtheorem{corollary}[theorem]{Corollary}

\title{Genus embeddings of complete graphs minus a matching}
\author{Timothy Sun\\Department of Computer Science\\San Francisco State University}
\date{}

\newcommand{\Z}{\mathbb{Z}}

\begin{document}

\maketitle

\begin{abstract}
We show that for all $n \equiv 0 \pmod{6}$, $n \geq 18$, there is an orientable triangular embedding of the octahedral graph on $n$ vertices that can be augmented with handles to produce a genus embedding of the complete graph of the same order. For these values of $n$, the intermediate embeddings of the construction also determine some surface crossing numbers of the complete graph on $n$ vertices and the genus of all graphs on $n$ vertices and minimum degree $n-2$. 
\end{abstract}

\section{Introduction}

Let $K_n$ denote the complete graph on $n$ vertices, and for even $n$, let $O_n = K_n-(n/2)K_2$ denote the octahedral graph on $n$ vertices. The orientable genus of the complete graphs \cite{Ringel-MapColor} and the octahedral graphs \cite{JungermanRingel-Octa, Sun-Index2} have the formulas
$$\gamma(K_n) = \left\lceil \frac{(n-3)(n-4)}{12} \right\rceil$$
and
$$\gamma(O_n) = \left\lceil \frac{(n-2)(n-6)}{12} \right\rceil,$$
which match a lower bound derived from Euler's polyhedral equation. For certain residues $n$ modulo $12$, all of the faces in these genus embeddings are triangular. It is perhaps not too surprising that such embeddings exist for sufficiently large $n$ given that these graphs have $\Theta(n^3)$ 3-cycles, while any triangular embedding would need only $\Theta(n^2)$ triangular faces. Mohar and Thomassen \cite{MoharThomassen} ask whether sufficiently dense graphs are guaranteed to have triangular embeddings:

\begin{problem}[Mohar and Thomassen \cite{MoharThomassen}]

Does there exist a constant $c \in (0, 1)$ such that every graph $G = (V,E)$ of minimum degree at least $c|V|$ satisfying $|E| \equiv 3|V| \pmod{6}$ triangulates an orientable surface?
\label{prob-mohar}
\end{problem}

While the author believes that their question has a positive answer, it is not close to being resolved in the affirmative. The present work makes progress on the modest case where the minimum degree is $n-2$.

Jungerman and Ringel \cite{JungermanRingel-Octa} tackled the $(n-2)$-regular case by constructing triangular embeddings of the octahedral graphs for $n \equiv 0, 2, 6, 8 \pmod{12}$ using ``orientable cascades,'' i.e., nonorientable current graphs with orientable derived embeddings. The remaining even residues $n \equiv 4, 10 \pmod{12}$ were later solved by Sun \cite{Sun-Index2}. The present work strengthens part of Jungerman and Ringel's result by showing that there exist triangular embeddings of $O_n$ for $n \equiv 0,6 \pmod{12}$ where a set of handles can be added to the embedding to obtain genus embeddings of the complete graphs $K_n$. 

In the original proof of the genus formula for the complete graphs \cite{Ringel-MapColor}, the residues $n \equiv 0,6 \pmod{12}$ had difficult constructions. In fact, the genus of $K_{18}$ and $K_{30}$ were among the last few special cases to be solved, and the $n \equiv 0 \pmod{12}$ case relied on current graphs over nonabelian groups. Simpler constructions for both of these residues are now known \cite{Sun-K12s, Sun-Minimum, Sun-nPrism, Sun-Kainen}. We provide yet another approach whose main advantage is that it also constructs genus embeddings of other graphs. 

The first known genus embeddings of $K_{18}$ were found essentially through trial and error \cite{Mayer-Orientables, Jungerman-K18}, but as shown in Sun \cite{Sun-Index2}, an easier construction results from augmenting the triangular embedding of $O_{18}$ of Jungerman and Ringel \cite{JungermanRingel-Octa} with two handles. We generalize this idea to all $n \equiv 0,6 \pmod{12}$, $n \geq 18$. Let $K(n,t)$ denote the complete graph with a matching of size $t \leq n/2$ deleted. For even $n$, the ends of this spectrum are the complete graphs $K_n = K(n,0)$ and the octahedral graphs $O_n = K(n,n/2)$. For these values of $n$, our constructions determine the genus of all such graphs $K(n, t)$, i.e., the graphs of minimum degree $2$:

\begin{theorem}
When $n \equiv 0 \pmod{6}$, the genus of $K(n,t)$, for $t = 0, \dotsc, n/2$, is
$$\gamma(K(n,t)) = \left\lceil \frac{(n-3)(n-4)-2t}{12}\right\rceil.$$
\label{thm-mindeg}
\end{theorem}

The \emph{surface crossing number} of a graph is the fewest number of crossings needed in a drawing of that graph in some specified closed surface. As a byproduct, the same construction provides surface crossing numbers of complete graphs over a range of surfaces:

\begin{theorem}
When $n \equiv 0 \pmod{6}$, the surface crossing number of $K_n$ in the surface of genus $\gamma(K_n)-g$, for $g = 1, \dotsc, \lceil n/12\rceil$, is $6g$ if $n \equiv 0 \pmod{12}$ and $6g-3$ if $n \equiv 6 \pmod{12}$.
\label{thm-crossing}
\end{theorem}

The proofs of these results are each divided into two cases: Theorem \ref{thm-mindeg} is proven as Corollaries \ref{corr-g6} and \ref{corr-g0}, and Theorem \ref{thm-crossing} is proven as Corollaries \ref{corr-c6} and \ref{corr-c0}.

\section{Background}

For more information on topological graph theory, see Gross and Tucker \cite{GrossTucker} or Mohar and Thomassen \cite{MoharThomassen}. 

Let $S_g$ denote the orientable surface of genus $g$. In this work, embeddings $\phi\colon G \to S_g$ of a graph $G = (V, E)$ in $S_g$ are usually \emph{cellular}, i.e., that $S_g \setminus \phi(G)$ is a disjoint union of open disks. The disks are called \emph{faces}, and we denote the set of faces as $F$.  Cellular embeddings satisfy the \emph{Euler polyhedral equation}
$$|V|-|E|+|F| = 2-2g.$$
Each face is bounded by a closed walk of length $k$. We call faces of length 3 and 4 \emph{triangular} and \emph{quadrangular}, respectively, and an embedding is triangular if all of its faces are triangular. The Euler polyhedral equation leads to well-known relationships between the numbers of edges and vertices of embedded graphs:
\begin{proposition}
If an embedding of a (not necessarily simple) graph $G = (V, E)$ in the surface $S_g$ is triangular, then
$$|E| = 3|V|-6+6g.$$
If a simple graph $G = (V,E)$ on at least 3 vertices is embedded (not necessarily cellularly) in the surface $S_g$, then
$$|E| \leq 3|V|-6+6g.$$
\label{prop-bound}
\end{proposition}

The \emph{genus} $\gamma(G)$ of a graph $G$ is the smallest integer $g$ such that $G$ can be embedded in $S_g$, and such an embedding is referred to as a \emph{genus embedding}. By rearranging the second part of Proposition \ref{prop-bound}, we obtain a result sometimes referred to as the \emph{Euler lower bound}:

\begin{corollary}
For a simple graph $G = (V,E)$ on at least 3 vertices, its genus is at least
$$\gamma(G) \geq \left\lceil \frac{|E|-3|V|+6}{6} \right\rceil.$$
\label{corr-euler}
\end{corollary}

We observe that the claim in Theorem \ref{thm-mindeg} and the genus formulas of the complete and octahedral graphs are equivalent to the fact that the the genus of these graphs match their Euler lower bound. In practice, these genus formulas are rarely referenced when constructing genus embeddings. Instead, it is more convenient to examine the distribution of faces in the embedding. If a simple graph has a triangular embedding, then that would automatically be a genus embedding. For graphs that cannot triangulate a surface due to the number of edges, one can check if only a few extra ``diagonals'' are needed to triangulate its nontriangular faces. 

\begin{proposition}
Let $G = (V,E)$ be a graph with a triangular embedding in some orientable surface. If $G'$ is any simple graph formed by deleting up to five edges from $G$, then the genus of $G'$ matches the Euler lower bound. 
\label{prop-del}
\end{proposition}
\begin{proof}
Let $G' = (V, E')$ be formed by deleting $m$ edges from $G$, for some nonnegative $m \leq 5$. Since $G$ has a triangular embedding, $|V| \geq 3$ and $(|E|-3|V|+6)/6$ is an integer. The embedding induced on $G'$ has genus
$$\frac{|E|-3|V|+6}{6} = \left\lceil\frac{|E|-3|V|+6}{6}-\frac{m}{6}\right\rceil = \left\lceil\frac{|E'|-3|V|+6}{6}\right\rceil.$$
This is exactly the Euler lower bound for $G'$. 
\end{proof}

Proposition \ref{prop-del} shows that Problem \ref{prob-mohar} has an equivalent formulation that applies to all graphs, not just those that can triangulate an orientable surface:

\begin{problem}
Does there exist a constant $c \in (0, 1)$ such that the genus of every graph $G = (V,E)$ of minimum degree at least $c|V|$ matches the Euler lower bound?
\label{prob-euler}
\end{problem}

The \emph{surface crossing number} $cr_g(G)$, introduced by Kainen \cite{Kainen-LowerBound}, is the fewest number of crossings needed to draw $G$ in $S_g$. Kainen proved a lower bound on the surface crossing number based on the girth of the graph, which we state here for girth $3$: 

\begin{proposition}[Kainen \cite{Kainen-LowerBound}]
Let $G = (V,E)$ be a simple graph on at least 3 vertices. The surface crossing number of $G$ in the surface of genus $g$ is at least 
$$cr_g(G) \geq |E|-(3|V|-6+6g).$$
\label{prop-kainen}
\end{proposition}
\begin{proof}
Deleting one edge from each crossing yields an embedding in the same surface, so the resulting graph is subject to Proposition \ref{prop-bound}. 
\end{proof}

Like the phenomenon suggested by Problem \ref{prob-euler}, Kainen's bound is often tight when the graph is dense and the genus of the surface is near the genus of the graph. In Sun \cite{Sun-Kainen}, it was shown that in the surface of genus $\gamma(K_n)-1$, for $n \not\equiv 0,3,4,7 \pmod{12}$ (i.e., the complete graphs that do not triangulate an orientable surface), the surface crossing number of the complete graph $K_n$ matches Kainen's lower bound, except when $n = 9$. In this work, we use octahedral graphs to calculate surface crossing numbers of complete graphs for a wider range of surfaces. To this end, we derive a more convenient form for complete graphs:

\begin{corollary}
If there is a triangular embedding of $K(n, t)$ in the surface $S_g$, then the surface crossing number of $K_n$ in $S_g$ is at least $t$. 
\label{corr-kainen}
\end{corollary} 
\begin{proof}
$cr_g(K_n) \geq \binom{n}{2} - \left(\binom{n}{2}-t\right) = t.$ In words, each missing edge needs to cross at least one other edge because the embedding is already triangular. 
\end{proof}

\subsection{Rotation systems}

Cellular embeddings in surfaces can be specified combinatorially in the following way. First, given an arbitrary orientation to the edges of the graph $G$, each edge $e \in E(G)$ induces two arcs $e^+$ and $e^-$ with the same endpoints but pointing in opposite directions. We use $E(G)^+$ to denote the set of such arcs. A \emph{rotation} at a vertex is a cyclic permutation of the arcs leaving the vertex. For a simple graph, it suffices to specify a cyclic permutation of the neighbors of that vertex. A \emph{(general) rotation system} assigns a rotation to each vertex and an \emph{edge signature} $\lambda\colon E(G) \to \{-1,1\}$. If an edge has signature $-1$, we say that it is \emph{twisted}. If the rotation system has no twisted edges, then it is \emph{pure} and the embedding is orientable. 

When tracing the faces of a rotation system, each face boundary walk begins in \emph{normal behavior}, where rotations are interpreted as, say, clockwise orderings of the edges leaving a vertex. If the walk traverses a twisted edge, the walk would then be in \emph{alternate behavior}, where the local orientation is reversed and rotations are now counterclockwise orderings. The walk reverts to normal behavior when it traverses another twisted edge. Given an orientation of each face boundary, an edge is said to be \emph{bidirectional} if the two times a face boundary walk traverses that edge are in opposite directions, and \emph{unidirectional}, otherwise. For a graph embedded with one face (like most of our current graphs), this property on edges is independent of the orientation of the face boundary walk. 

A \emph{vertex flip} reverses the rotation at that vertex and switches the signatures of each incident edge (unless it is a self-loop). Vertex flips do not affect the set of faces, and hence two rotation systems are considered to be equivalent if there is a sequence of vertex flips that transforms one rotation system into the other. A rotation system describes an orientable embedding if and only if it can be transformed into a pure rotation system by a sequence of vertex flips. Conversely, if there is a closed walk with an odd number of twisted edges, then no such sequence exists and the embedding is nonorientable. 

\subsection{Current graphs}

A \emph{current graph} $(\phi, \alpha)$ consists of a cellular embedding $\phi\colon G \to S$ in a possibly nonorientable surface $S$ and an arc-labeling $\alpha\colon E(G)^+ \to \Gamma$ satisfying $\alpha(e^+) = -\lambda(e)\alpha(e^-)$ for every edge $e$. The group $\Gamma$ is called the \emph{current group} and the arc labels are referred to as \emph{currents}. In this paper, the current group is always the (additive) cyclic group $\Z_n$, where $n$ is a multiple of $6$. The \emph{excess} of a vertex is the sum of all of the currents on the arcs entering the vertex, and if the excess is 0, we say that that vertex satisfies \emph{Kirchhoff's current law}. 

Face boundary walks $(e^\pm_1, e^\pm_2, e^\pm_3, \dotsc)$ are called \emph{circuits}, and the \emph{log} of a circuit replaces each $e^\pm_i$ with $\alpha(e^\pm_i)$ or $-\alpha(e^\pm_i)$ if the walk is in normal or alternate behavior, respectively. With one exception, the current graphs in this work are \emph{orientable cascades}, i.e., current graphs with a 1-face embedding in a nonorientable surface whose derived embeddings are orientable. Paraphrasing from Jungerman and Ringel \cite{JungermanRingel-Octa}, these orientable cascades satisfy a standard set of properties: 
\begin{enumerate}
\item[(C1)] Each vertex is of degree 1 or 3.
\item[(C2)] There is one face in the embedding, inducing a single circuit. 
\item[(C3)] Kirchhoff's current law is satisfied at every vertex of degree 3.
\item[(C4)] The excess of a degree 1 vertex has order 3 in $\Z_n$.
\item[(C5)] The log of the circuit contains each element of $\Z_n\setminus\{0, n/2\}$ exactly once.
\item[(C6)] The current on an edge is even if and only if the edge is bidirectional.
\end{enumerate}
For example, consider the orientable cascade in Figure \ref{fig-ex} with current group $\Z_{18}$ satisfying these properties. Solid and hollow vertices represent clockwise and counterclockwise rotations, respectively. It is identical to one described in Jungerman and Ringel \cite{JungermanRingel-Octa}, except with all of the vertices flipped. If its lone circuit is oriented so that it is in normal behavior as it traverses the degree 1 vertex, then the log is:
\begin{figure}[tbp]
\centering
\includegraphics[scale=1]{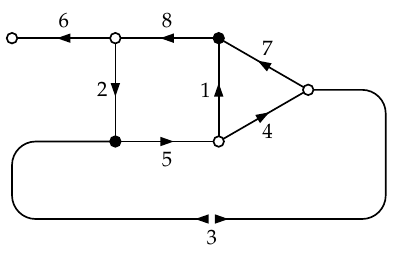}
\caption{An orientable cascade that generates a triangular embedding of $O_{18}$.}
\label{fig-ex}
\end{figure}
$$\begin{array}{ccccccccccccccccccccccccccccc}
(6 & 12 & 2 & 5 & 4 & 15 & 13 & 17 & 7 & 3 & 16 & 10 & 11 & 14 & 1 & 8).
\end{array}$$
Each current graph generates a \emph{derived embedding} of a \emph{derived graph}. The derived graph has vertex set $\Z_n$, and the rotation at each vertex $i \in \Z_n$, and hence the neighborhood of $i$, is determined by adding $i$ to each element of the log. This is sometimes referred to as the \emph{additivity rule}. Normally, for a current graph with twisted edges, one would also need to specify the signatures of the edges in the derived embedding. However, because of property (C6), the twisted edges are exactly those with one odd and one even endpoint (see Section 4.1.6 of Gross and Tucker \cite{GrossTucker} or Section 8.4 of Ringel \cite{Ringel-MapColor}). By flipping, say, the odd vertices, we obtain a pure rotation system, demonstrating that the derived embedding is orientable. While most of the current graphs in this paper are embedded in nonorientable surfaces, all derived embeddings will be orientable, so we often omit mentioning the orientability of such an embedding. 

Applying this process to the above log results in a triangular embedding of the octahedral graph $O_{18}$, since property (C5) implies that each vertex $i$ is adjacent to all other vertices except $i+9$. It was shown in Sun \cite{Sun-Index2} how to augment this embedding with handles to produce a genus embedding of $K_{18}$. We now describe this construction and its generalization to all larger $n \equiv 0 \pmod{6}$. 

\section{The constructions}

In Jungerman and Ringel \cite{JungermanRingel-Octa}, triangular embeddings of the octahedral graphs on $n \equiv 0 \pmod{6}$ vertices are generated by orientable cascades. We require new families satisfying additional properties so that we can attach handles that introduce the missing edges. The handles and edge flips differ between the $n \equiv 0 \pmod{12}$ and $n \equiv 6 \pmod{12}$ cases, so we describe them separately. 

\subsection{$n \equiv 6 \pmod{12}$}

\begin{theorem}
For $s \geq 1$, there is an orientable triangular embedding of the octahedral graph $O_{12s+6}$ that can be augmented with $s+1$ handles to obtain a genus embedding of the complete graph $K_{12s+6}$. 
\label{thm-case6}
\end{theorem}
\begin{proof}
Consider the current graphs in Figures \ref{fig-ex} and \ref{fig-case6} with current group $\Z_{12s+6}$, for each $s \geq 1$. In our infinite families, the ellipses denote a ``ladder,'' such as the one shown in Figure \ref{fig-ladder}, where the currents on the vertical ``rungs'' form a sequence of consecutive integers, the directions of those arcs alternate, and the rotations on the vertices form a checkerboard pattern. The currents on the horizontal arcs can be deduced from Kirchhoff's current law. 

\begin{figure}[tbp]
\centering
    \begin{subfigure}[b]{0.99\textwidth}
        \centering
        \includegraphics[scale=0.9]{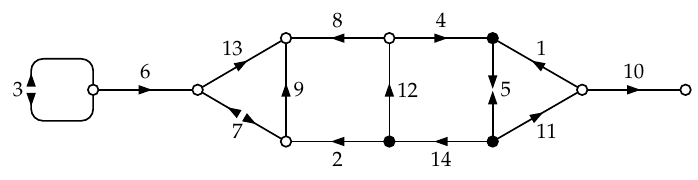}
        \caption{}
    \end{subfigure}
    \begin{subfigure}[b]{0.99\textwidth}
        \centering
        \includegraphics[scale=0.9]{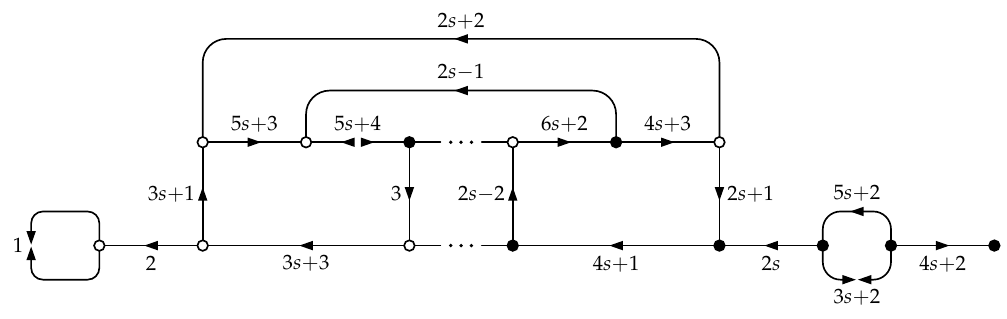}
        \caption{}
    \end{subfigure}
    \begin{subfigure}[b]{0.99\textwidth}
        \centering
        \includegraphics[scale=0.9]{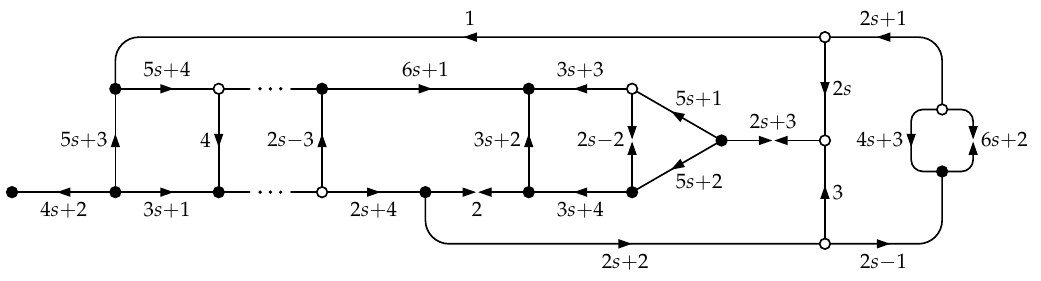}
        \caption{}
    \end{subfigure}
\caption{Current graphs for $O_{12s+6}$ for (a) $s = 2$, (b) odd $s \geq 3$, and (c) even $s \geq 4$.}
\label{fig-case6}
\end{figure}

\begin{figure}[tbp]
\centering
\includegraphics[scale=1]{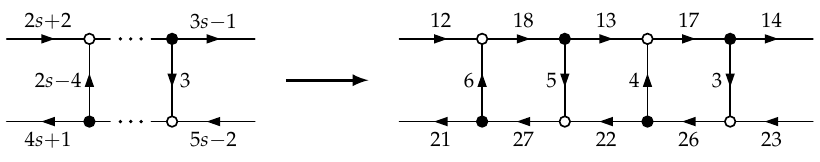}
\caption{A ladder and its specification for $s = 5$.}
\label{fig-ladder}
\end{figure}

For readability, we set $r = 2s+1 = (12s+6)/6$ and describe just one of the handles---the other handles will be obtained via the additivity rule: we add $2i$ to each vertex that is mentioned, for appropriate values of $i \in \Z_{12s+6}$. In each current graph, we orient the circuit so that it is in normal behavior as it traverses the degree 1 vertex. Thus, that vertex induces the faces $[0, 4r, 2r]$ and $[r, 3r, 5r]$ (the latter face has the ``opposite'' orientation since $r$ is odd). We connect these two faces with a handle to add the edges $(0, 3r)$, $(2r, 5r)$, $(4r, 2r)$, $(0, r)$, $(2r, 3r)$, and $(4r, 5r)$, as illustrated on the left of Figure \ref{fig-add6}(a). The latter three edges already exist in the graph, so we remove their original instances. In the resulting quadrangular faces, as seen on the right of Figure \ref{fig-add6}(a), the other diagonals are the missing edges $(q, q+3r)$, $(q+2r, q+5r)$, and $(q+4r, q+r)$, where $q$ is $2r+1$ and $1$ for odd and even $s$, respectively. After adding the handle, the embedding remains triangular.

\begin{figure}[tbp]
\centering
    \begin{subfigure}[b]{0.45\textwidth}
        \centering
        \includegraphics[scale=1]{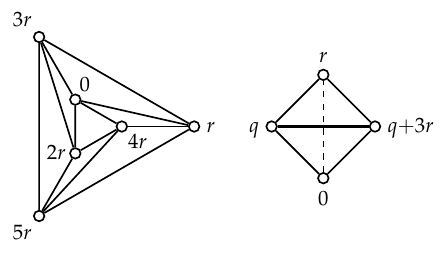}
        \caption{}
    \end{subfigure}
    \begin{subfigure}[b]{0.25\textwidth}
        \centering
        \includegraphics[scale=1]{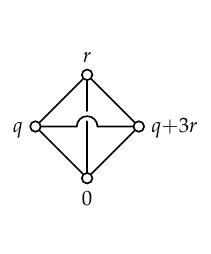}
        \caption{}
    \end{subfigure}
\caption{A handle and an edge flip used to incorporate missing edges.}
\label{fig-add6}
\end{figure}

In each case, the six edges added by this handle are of the form $(j, j+3r)$ for $j = 0, 1, 2r, 2r+1, 4r, 4r+1$. Thus, we can sequentially apply this handle-addition process for each $i = 0, \dotsc, s-1$ (with $i$ defined above) while maintaining that the sets of added edges are disjoint. When we try to apply this procedure one more time by setting $i = s$, this final handle overlaps with the $i = 0$ handle at the three edges $(0, 3r)$, $(r, 4r)$, and $(2r, 5r)$ (e.g., $(i,j) = (0, 0)$ and $(i,j) = (s, 2r+1)$ are the same edge). Thus, we obtain a triangular embedding of $K_{12s+6}$ with three extra edges. Figure \ref{fig-schematic} illustrates the case $s = 2$, where the overlap is illustrated by dashed lines in the rightmost handle. By Proposition \ref{prop-del}, deleting those extra edges results in a genus embedding of $K_{12s+6}$. 

\begin{figure}[tbp]
\centering
\includegraphics[scale=1]{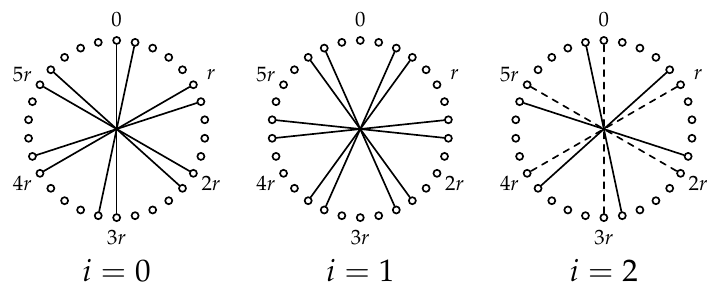}
\caption{Contributions from each handle for $O_{30}$.}
\label{fig-schematic}
\end{figure}
\end{proof}

The sequence of embeddings in the above construction yields genus embeddings for all graphs on $n = 12s+6$ vertices with minimum degree $n-2$:

\begin{corollary}
For $s \geq 0$, there exist orientable triangular embeddings of $K(12s+6, 6s+3-6t)$ for $t = 0, \dotsc, s-1$. Hence, the genus of $K(12s+6, k)$ for $k = 0, \dotsc, 6s+3$, matches the Euler lower bound. 
\label{corr-g6}
\end{corollary}
\begin{proof}
Besides the triangular embeddings in the proof of Theorem \ref{thm-case6}, there is also a triangular embedding of $K(6,3)$ in the plane, and the genus of $K_6$ is $1$ \cite{Ringel-MapColor}.

If $k \geq 3$, then define $k'$ to be the largest integer where $k' \leq k$ and $k' \equiv 3 \pmod{6}$. Then, we can apply Proposition \ref{prop-del} to the triangular embedding of $K(12s+6, k')$ to determine the genus of $K(12s+6, k)$. For $k = 1,2$, the Euler lower bound for $K(12s+6,k)$ is the same as that of $K_{12s+6}$. Alternatively, we can apply Proposition \ref{prop-del} again to the final embedding with duplicate edges.
\end{proof}

The edge flips used in the construction imply tight bounds on surface crossing numbers:

\begin{corollary}
For $s \geq 0$, the surface crossing number of $K_{12s+6}$ in the surface of genus $\gamma(K_{12s+6})-h$, for $h = 1, \dotsc, s+1$ is $6h-3$. 
\label{corr-c6}
\end{corollary}
\begin{proof}
For $s = 0$, the planar crossing number of $K_6$ is $3$ \cite{HararyHill}. For larger $s$, each missing edge $e_i = (q+2i, q+3r+2i)$ can be incorporated into the triangular embedding of $O_{12s+6}$ by crossing the edge $(2i, r+2i)$, like in Figure \ref{fig-add6}(b). These crossings do not interfere with each other since the edges are added inside of disjoint pairs of faces. We may still insert $e_i$ in this manner in any subsequent triangular embedding in the proof of Theorem \ref{thm-case6} where $e_i$ is still missing. Since each missing edge is incorporated using a single crossing, the total number of crossings match the lower bound in Corollary \ref{corr-kainen}. 
\end{proof}

\subsection{$n \equiv 0 \pmod{12}$}

Once again, we utilize the degree 1 vertex with order 3 excess, but unlike in the previous residue, the process is complicated by the fact that missing edges are between vertices of the same parity. As a result, an extra set of edge flips are needed, but we have increased flexibility in deciding which pairs of faces we connect with handles. 

\begin{theorem}
For $s \geq 2$, there is an orientable triangular embedding of the octahedral graph $O_{12s}$ that can be augmented with $s$ handles to obtain a triangular embedding of the complete graph $K_{12s}$. 
\end{theorem}

\begin{proof}
Consider the current graphs in Figures \ref{fig-case0-s2} and \ref{fig-case0} with current group $\Z_{12s}$, defined for all $s \geq 2$. For the $s = 2$ case in Figure \ref{fig-case0-s2}, it is an orientable index 2 current graph. Instead of describing all the properties needed to explain this single case, we just list out its logs:
$$\arraycolsep=4pt\begin{array}{rlccccccccccccccccccccl}
\lbrack0\rbrack. & (8 & 16 & 23 & 22 & 3 & 5 & 2 & 1 & 14 & 18 & 15 & 21 & 6 & 20 & 13 & 17 & 4 & 10 & 11 & 19 & 9 & 7) \\
\lbrack1\rbrack. & (16 & 8 & 13 & 23 & 1 & 17 & 2 & 21 & 19 & 22 & 15 & 10 & 6 & 9 & 3 & 18 & 4 & 11 & 7 & 20 & 14 & 5).
\end{array}$$
To generate the rotation at vertex $i \in \Z_{24}$, take the log of circuit $[i \bmod{2}]$ and add $i$ to each element. One can check that the resulting rotation system is triangular. For more information on index 2 current graphs and their relationship with orientable cascades, see Sun \cite{Sun-Index2}. 

\begin{figure}[tbp]
\centering
\includegraphics[scale=0.9]{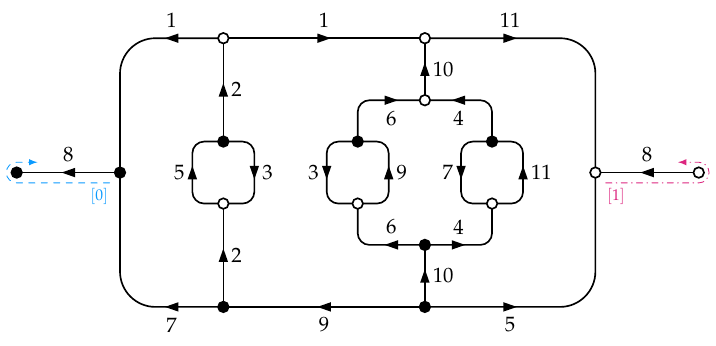}
\caption{An index 2 current graph for $O_{24}$.}
\label{fig-case0-s2}
\end{figure}

\begin{figure}[tbp]
\centering
    \begin{subfigure}[b]{0.99\textwidth}
        \centering
        \includegraphics[scale=0.9]{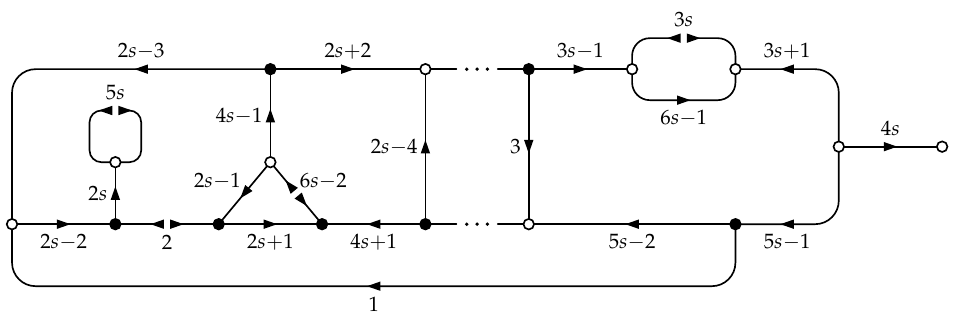}
        \caption{}
    \end{subfigure}
    \begin{subfigure}[b]{0.99\textwidth}
        \centering
        \includegraphics[scale=0.9]{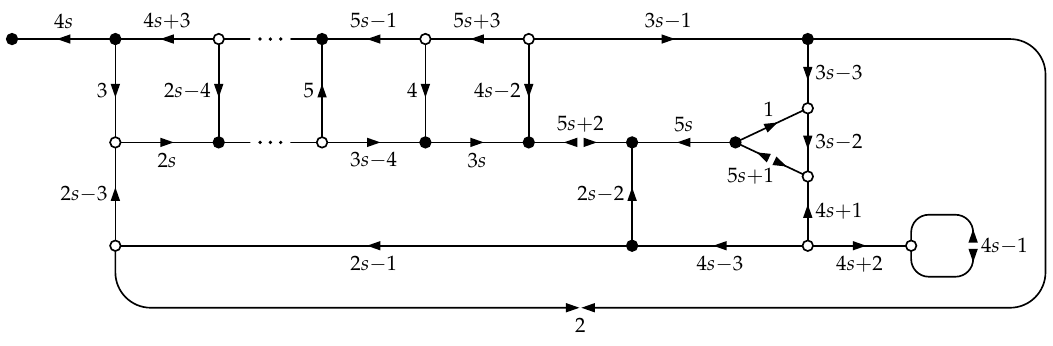}
        \caption{}
    \end{subfigure}
\caption{Current graphs for $O_{12s}$ for (a) odd $s \geq 3$, and (b) even $s \geq 4$.}
\label{fig-case0}
\end{figure}

Like before, we illustrate just one of the handles and let $r = 2s = (12s)/6$. We connect the faces $[0, 4r, 2r]$ and $[a, a+2r, a+4r]$, where $a$ is some odd number. As seen in Figure \ref{fig-add0}, we move two groups of existing edges into this handle: edges of the form $(j,a+j)$ and $(j,a+2r+j)$, where $j = 0, 2r, 4r$. In the resulting quadrilaterals, we can add the missing edges $(b+j, b+3r+j)$ and $(c+j, c+3r+j)$, where $b$ and $c$ are numbers of opposite parity. We obtain the other handles by adding $2i$ to each vertex in this construction, for $i = 1, \dotsc, r-1$. Since $b$ and $c$ are of opposite parities, all of these added edges are distinct. 

\begin{figure}[tbp]
\centering
\includegraphics[scale=1]{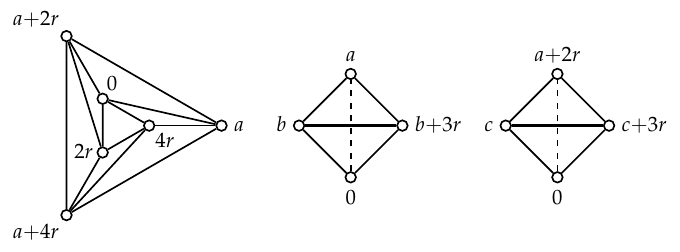}
\caption{A handle and the two kinds of edge flips used to incorporate missing edges.}
\label{fig-add0}
\end{figure}

The values $(a,b,c)$ are $(1, 2, 19)$ for $s = 2$, $(2s-1, 8s+1, 3s-1)$ for odd $s \geq 3$, and $(12s-3, 10s, 8s-1)$ for even $s \geq 4$. 
\end{proof}

Like in the $n \equiv 6 \pmod{12}$ case, the construction has additional benefits:

\begin{corollary}
For $s \geq 1$, there exist orientable triangular embeddings of $K(12s, 6s-6t)$ for $t = 0, \dotsc, s$. Hence, the genus of $K(12s, k)$ for $k = 0, \dotsc, 6s$, matches the Euler lower bound. 
\label{corr-g0}
\end{corollary}
\begin{corollary}
For $s \geq 1$, the surface crossing number of $K_{12s}$ in the surface of genus $\gamma(K_{12s})-t$, for $t = 0, \dotsc, s$ is $6t$. 
\label{corr-c0}
\end{corollary}
\begin{proof}
The only remaining case for both results is when $s = 1$. There exists a triangular embedding of $K_{12}$ \cite{Ringel-MapColor}. In Jungerman and Ringel \cite{JungermanRingel-Octa}, the orientable cascade generating a triangular embedding of $O_{12}$ has a log of the form
$$\begin{array}{rcccccccccccccccccccccl}
(\dotsc & 2 & 11 & 8 & \dotsc),
\end{array}$$
so the missing edge $(2,8)$ can be added to graph, crossing the edge $(0,11)$. Similarly, by the additivity rule, each of the other five missing edges can also be added to the drawing using one crossing (though for the edges between odd vertices, the orientation is reversed). Thus the surface crossing number of $K_{12}$ in the surface $S_5$ is $6$. 
\end{proof}

\section{Conclusion}

We described a construction for transforming genus embeddings of octahedral graphs into those of complete graphs, when the number of vertices is a multiple of 6. The intermediate embeddings generated in the construction provides some evidence in favor of an affirmative answer to Mohar and Thomassen's Problem \ref{prob-mohar}. Additional, they determine some surface crossing numbers of the complete graphs. 

Jungerman and Ringel \cite{JungermanRingel-Minimal} found a \emph{minimal triangulation} for each surface $S_g$, $g \neq 2$. In particular, for each such surface, they constructed a triangular embedding of a simple graph on $$M(g) := \left\lceil \frac{7+\sqrt{1+48g}}{2} \right\rceil$$ vertices, and this is the fewest number of vertices possible. Previously, the author \cite{Sun-Kainen} conjectured a generalization of this result, that Kainen's lower bound is tight for the complete graph on $M(g)$ vertices in the surface $S_g$, i.e., that there is a minimal triangulation of $S_g$ where each missing edge can be added to the embedding using only one crossing each. Theorem \ref{thm-crossing} solves this conjecture for some surfaces, but minimal triangulations can have as many as $n-6$ fewer edges than the complete graph $K_n$. 

Our constructions exploit the vertex of degree 1, but the current graphs used to find the genus of the other octahedral graphs \cite{JungermanRingel-Octa, Sun-Index2} do not have this feature. Are there other simple structures that allow for a similar kind of handle operation? Furthermore, what about the graphs $K(n,t)$ for $n$ odd? At present, the genus of such graphs, especially for large $t$, is almost completely unknown. Even for small $t$, there are few results besides those derived from Proposition \ref{prop-del}: Su, Noguchi, and Zhou \cite{Su} showed that $\gamma(K(9,4)) = \gamma(K_9) = 3$, extending a result of Huneke \cite{Huneke-Minimum}. Perhaps the odd $n$ case is the most promising and fertile direction to explore next. 

\bibliographystyle{alpha}
\bibliography{biblio}

\begin{thebibliography}{Sun24b}

\bibitem[GT87]{GrossTucker}
Jonathan~L. Gross and Thomas~W. Tucker.
\newblock {\em {Topological Graph Theory}}.
\newblock John Wiley \& Sons, 1987.

\bibitem[HH63]{HararyHill}
Frank Harary and Anthony Hill.
\newblock On the number of crossings in a complete graph.
\newblock {\em Proceedings of the Edinburgh Mathematical Society},
  13(4):333--338, 1963.

\bibitem[Hun78]{Huneke-Minimum}
John~Philip Huneke.
\newblock A minimum-vertex triangulation.
\newblock {\em Journal of Combinatorial Theory, Series B}, 24(3):258--266,
  1978.

\bibitem[JR78]{JungermanRingel-Octa}
Mark Jungerman and Gerhard Ringel.
\newblock The genus of the $n$-octahedron: {Regular cases}.
\newblock {\em Journal of Graph Theory}, 2(1):69--75, 1978.

\bibitem[JR80]{JungermanRingel-Minimal}
Mark Jungerman and Gerhard Ringel.
\newblock Minimal triangulations on orientable surfaces.
\newblock {\em Acta Mathematica}, 145(1):121--154, 1980.

\bibitem[Jun74]{Jungerman-K18}
Mark Jungerman.
\newblock {Orientable triangular embeddings of $K_{18}-K_3$ and $K_{13}-K_3$}.
\newblock {\em Journal of Combinatorial Theory, Series B}, 16(3):293--294,
  1974.

\bibitem[Kai72]{Kainen-LowerBound}
Paul~C. Kainen.
\newblock A lower bound for crossing numbers of graphs with applications to
  {$K_n$, $K_{p,q}$, and $Q(d)$}.
\newblock {\em Journal of Combinatorial Theory, Series B}, 12(3):287--298,
  1972.

\bibitem[May69]{Mayer-Orientables}
Jean Mayer.
\newblock Le probleme des r{\'e}gions voisines sur les surfaces closes
  orientables.
\newblock {\em Journal of Combinatorial Theory}, 6(2):177--195, 1969.

\bibitem[MT01]{MoharThomassen}
Bojan Mohar and Carsten Thomassen.
\newblock {\em Graphs on {S}urfaces}, volume~10.
\newblock Johns Hopkins University Press, 2001.

\bibitem[Rin74]{Ringel-MapColor}
Gerhard Ringel.
\newblock {\em {Map Color Theorem}}, volume 209.
\newblock Springer Science \& Business Media, 1974.

\bibitem[SNZ15]{Su}
Huadong Su, Kenta Noguchi, and Yiqiang Zhou.
\newblock Finite commutative rings with higher genus unit graphs.
\newblock {\em Journal of Algebra and Its Applications}, 14(01):1550002, 2015.

\bibitem[Sun19]{Sun-K12s}
Timothy Sun.
\newblock A simple construction for orientable triangular embeddings of the
  complete graphs on $12s$ vertices.
\newblock {\em Discrete Mathematics}, 342(4):1147--1151, 2019.

\bibitem[Sun20]{Sun-Minimum}
Timothy Sun.
\newblock Simultaneous current graph constructions for minimum triangulations
  and complete graph embeddings.
\newblock {\em Ars Mathematica Contemporanea}, 18(2):309--337, 2020.

\bibitem[Sun23]{Sun-nPrism}
Timothy Sun.
\newblock Settling the genus of the $n$-prism.
\newblock {\em European Journal of Combinatorics}, 110:103667, 2023.

\bibitem[Sun24a]{Sun-Index2}
Timothy Sun.
\newblock Jungerman ladders and index 2 constructions for genus embeddings of
  dense regular graphs.
\newblock {\em European Journal of Combinatorics}, 120:103974, 2024.

\bibitem[Sun24b]{Sun-Kainen}
Timothy Sun.
\newblock On {K}ainen's conjectures on surface crossing numbers.
\newblock {\em arXiv preprint arXiv:2405.06118}, 2024.

\end{thebibliography}

\end{document}